\documentclass[12pt]{article}

\usepackage{amsmath}
\usepackage{amsfonts}
\usepackage{color} 
\usepackage{amssymb}
\usepackage{graphicx}
\usepackage{hyperref}
\usepackage{enumerate}
\usepackage[all]{xy}
\usepackage{arydshln}

\def\tr{{\rm tr}}

 \allowdisplaybreaks

\begin{document}

\newtheorem{problem}{Problem}

\newtheorem{theorem}{Theorem}[section]
\newtheorem{corollary}[theorem]{Corollary}
\newtheorem{definition}[theorem]{Definition}
\newtheorem{conjecture}[theorem]{Conjecture}
\newtheorem{question}[theorem]{Question}
\newtheorem{lemma}[theorem]{Lemma}
\newtheorem{proposition}[theorem]{Proposition}
\newtheorem{quest}[theorem]{Question}
\newtheorem{example}[theorem]{Example}

\newenvironment{proof}{\noindent {\bf
Proof.}}{\rule{2mm}{2mm}\par\medskip}

\newenvironment{proofof3}{\noindent {\bf
Proof of  Theorem 1.2.}}{\rule{2mm}{2mm}\par\medskip}

\newenvironment{proofof5}{\noindent {\bf
Proof of  Theorem 1.3.}}{\rule{2mm}{2mm}\par\medskip}

\newcommand{\remark}{\medskip\par\noindent {\bf Remark.~~}}
\newcommand{\pp}{{\it p.}}
\newcommand{\de}{\em}

\title{{Extensions of some matrix inequalities related to trace and partial traces}\thanks{This paper was firstly announced in March, 2020, and was later published on Linear Algebra and its Applications 639 (2022) 205--224. 
See \url{https://doi.org/10.1016/j.laa.2022.01.006}. E-mail addresses: ytli0921@hnu.edu.cn (Y\v{o}ngt\={a}o L\v{i}).}
 }

\author{Yongtao Li$^*$\\
{\small School of Mathematics, Hunan University} \\
{\small Changsha, Hunan, 410082, P.R. China }  }

\maketitle

 \begin{center}
Dedicated to Prof. Weijun Liu  on his 60th birthday
 \end{center}

\begin{abstract}
We first present a determinant inequality related to partial traces  
for positive semidefinite block matrices. 
Our result extends  
a  result of Lin [Czech. Math. J. 66 (2016)] 
 and improves 
 a result of Kuai [Linear Multilinear Algebra 66 (2018)].  
Moreover, we provide a unified treatment of 
a result of Ando [ILAS Conference (2014)] and a recent 
result of Li, Liu and Huang [Operators and Matrices 15 (2021)]. 
Furthermore, we also extend some determinant inequalities involving partial traces 
to a larger class of matrices 
whose numerical ranges are contained in a sector. 
In addition, some extensions on trace inequalities for positive semidefinite $2\times 2$ block matrices are also included. 
 \end{abstract}



\section{Introduction}

\label{sec1} 

Throughout the paper, we use the following standard notation. 
The set of $n\times n$ complex matrices is denoted by $\mathbb{M}_n(\mathbb{C})$, 
or simply by $\mathbb{M}_n$, 
and the identity matrix of order $n$ by  $I_n$, or $I$ for short. 
We write $\lambda_i (A)$ and $\sigma_i(A)$ 
for the $i$-th largest eigenvalue and 
singular value of $A$, respectively.   
By convention, if $A\in \mathbb{M}_n$ is positive semidefinite, 
we write $A\ge 0$. For  Hermitian matrices $A$ and $B$ with the same size, 
$A\ge B$ means that  $A-B$ is positive semidefinite, i.e., $A-B\ge 0$.  
If $A=[a_{i,j}] $ is of order $m\times n$ 
and $B$ is of order $s\times t$, the tensor product of $A$ with $B$, 
denoted by $A\otimes B$, is an $ms\times nt$ matrix that  
partitioned into $m\times n$ block matrices with the $(i,j)$-block being the $s\times t$ matrix $a_{i,j}B$.
In this paper, 
we are interested in complex block matrices. Let $\mathbb{M}_n(\mathbb{M}_k)$ 
be the set of complex matrices partitioned into $n\times n$ blocks 
with each block being $k\times k$. 
The element of $\mathbb{M}_n(\mathbb{M}_k)$ is usually written as ${ H}=[H_{i,j}]_{i,j=1}^n$, 
where $H_{i,j}\in \mathbb{M}_k$ for all $i,j$.

Now we introduce the definition of partial traces, 
which comes from Quantum Information Theory \cite[p. 12]{Petz08}.  
For $H\in \mathbb{M}_n(\mathbb{M}_k)$, 
the first partial trace (map) $H \mapsto \mathrm{tr}_1 H \in \mathbb{M}_k$ is defined as the  
adjoint map of the embedding map $X \mapsto I_n\otimes X\in \mathbb{M}_n\otimes \mathbb{M}_k$. 
Correspondingly, the second partial trace (map)  $H \mapsto \mathrm{tr}_2 H\in \mathbb{M}_n$ is 
defined as the adjoint map of the embedding map 
$Y\mapsto Y\otimes I_k \in \mathbb{M}_n\otimes \mathbb{M}_k$. Therefore, we have
\begin{equation} \label{eqdef} 
\langle I_n\otimes X, H \rangle =\langle X, \mathrm{tr}_1H \rangle ,
\quad \forall X\in \mathbb{M}_k, 
\end{equation}
and 
\[ \langle Y\otimes I_k, H \rangle =\langle Y,\mathrm{tr}_2 H \rangle, 
\quad \forall Y\in \mathbb{M}_n, \]
where $\langle \cdot ,\cdot \rangle$ stands for the Hilbert-Schmidt inner product, i.e., 
$\langle A,B \rangle =\tr (A^*B)$. 
The above definition of partial traces is implicit. 
Assume that $H=[H_{i,j}]_{i,j=1}^n$ is an $n\times n$ block 
matrix with $H_{i,j}\in \mathbb{M}_k$, 
the visualized version of the  partial traces 
is equivalently given in  \cite[pp. 120--123]{Bh07} as
\begin{equation} \label{eqdef2}
 \mathrm{tr}_1 { H}=\sum\limits_{i=1}^n H_{i,i} , 
 \end{equation}
 and 
 \begin{equation*}
\mathrm{tr}_2{ H}=\bigl[ \mathrm{tr}H_{i,j}\bigr]_{i,j=1}^n. 
\end{equation*}
It is easy to see that  both $\tr_1 H$ and $\tr_2 H$ are positive semidefinite 
whenever ${ H} \in \mathbb{M}_n(\mathbb{M}_k)$ is positive semidefinite; see, e.g. \cite[p. 237]{Zhang11} or \cite{Zha12} 
for more details. 
The first or second partial trace is a source for matrix inequalities 
and  extensively studied in recent years; 
see \cite{Ando14,Choi18,FLT20,HuangLi20,Lin16} for related topics.

Let $A=[A_{i,j}]_{i,j=1}^n$ be an $n\times n$ block matrix with each block 
being a $k\times k$ matrix. 
The usual transpose of $A$ is defined as 
$A^T= [A_{j,i}^T]_{i,j=1}^n$. 
We define the {\it partial transpose} of $A$ by $A^{\tau}=[A_{j,i}]_{i,j=1}^n$, that is, the partial transpose of 
$A$ is the matrix obtained by 
transposing blocks of $A$ independently. More precisely,   
\[  A^T = \begin{bmatrix} 
A_{1,1}^T & \cdots  & A_{n,1}^T \\
\vdots & \ddots & \vdots \\
A_{1,n}^T & \cdots & A_{n,n}^T
\end{bmatrix} ~~\text{and}~~ 
A^{\tau} = \begin{bmatrix} 
A_{1,1} & \cdots  & A_{n,1} \\
\vdots & \ddots & \vdots \\
A_{1,n} & \cdots & A_{n,n}
\end{bmatrix} . \] 
Although $A$ and $A^{\tau}$ have the same trace, 
they may have different eigenvalues, so 
they are not necessarily similar.  Moreover, 
it is known that $A\ge 0$ does not necessarily imply $A^{\tau}\ge 0$. 
For example, taking 
\begin{equation} \label{eqeq1}
 A=\begin{bmatrix} 
A_{1,1} & A_{1,2} \\ A_{2,1} & A_{2,2} \end{bmatrix}=
\left[\begin{array}{cc;{2pt/2pt}cc}
1&0&0&1\\
0&0&0&0\\
\hdashline[2pt/2pt] 
0&0&0&0\\
1&0&0&1 \end{array}\right] .
\end{equation}
We can see  from the definition that 
\[ A^{\tau}=\begin{bmatrix} 
A_{1,1} & A_{2,1} \\ A_{1,2} & A_{2,2} \end{bmatrix}=
\left[\begin{array}{cc;{2pt/2pt}cc}
1&0&0&0\\
0&0&1&0\\
\hdashline[2pt/2pt] 
0&1&0&0\\
0&0&0&1 \end{array}\right]. \]
One could easily observe that $A$ is positive semidefinite,  
but $A^{\tau}$ is not positive semidefinite 
since it contains a principal submatrix 
$\left[\begin{smallmatrix}0 & 1 \\ 1 &0 \end{smallmatrix}\right] \ngeq 0$. 
Moreover, the eigenvalues of $A$ are $2,0,0,0$, 
and the eigenvalues of $A^{\tau}$ are $1,1,1,-1$, 
so $A$ and $A^{\tau}$ are not similar. 
In addition, replacing $A_{1,1}$ in the above matrix by 
$\left[\begin{smallmatrix}1 & 0 \\ 0 &1 \end{smallmatrix}\right]$ 
also gives a well example. 
From this discussion, we say that 
$A$ is  {\it positive partial transpose} (or PPT for short) if 
 both $A$ and $A^{\tau}$ are positive semidefinite. 
 We recommend  
\cite{FLTmia,Lee2015,Lin14,Linoam2015} for recent progress.

The paper is organized as follows.   
In Section \ref{sec2}, we shall review some preliminaries 
for a  class of matrices 
whose numerical ranges are contained in a sector (known as the sector matrices). 
This is a natural extension of the class of positive definite matrices. 
In Section \ref{sec3}, 
we shall study the recent results involving 
 the Fiedler--Markham inequality. 
We provide an  extension of a result of Lin \cite{Lin16}, 
and our result is also an improvement of a result of Kuai \cite{Kua17}; 
see Theorem \ref{thm21}. 
Moreover, we shall  extend a result of Choi  \cite{Choi17} 
to the so-called sector matrices; 
see Theorem \ref{thm26}. 
In Section \ref{sec4}, we give  a unified treatment of 
a result of Ando \cite{Ando14} (or see \cite{Lin16b}) 
as well as a recent result of Li, Liu and Huang  \cite{HuangLi20}.  
Our new treatment is more concise than original proof.  
Moreover, we also present some Ando type determinant inequalities 
for  partial traces, and then we extend these inequalities
 to  sector matrices; see Theorems \ref{prop44} 
 and  \ref{thm36}. 
In Section \ref{sec5}, we shall prove some inequalities 
for positive semidefinite $2\times 2$ block matrices; 
see Theorems \ref{thm52}, \ref{thm53} and \ref{thm54}. 
Our result extend slightly the recent elegant work 
on trace inequalities 
that proved by  Kittaneh and Lin \cite{KL17} and Lin \cite{Lin14} as well.

\section{Preliminaries}
\label{sec2}

Recall that $\sigma_i(A)$ denotes  $i$-th largest singular value of $A$. 
When $A$ is Hermitian, we know that all eigenvalues of $A$ are real numbers,  
and we write  $\lambda_i(A)$ for the $i$-th largest eigenvalue. 
The numerical range of $A\in \mathbb{M}_n$ is defined by 
\[ W(A)=\{x^*Ax : x\in \mathbb{C}^n,x^*x=1\}. \]
For $\alpha \in [0,{\pi}/{2})$, let $S_{\alpha}$ be the sector on  complex plane defined as  
\[ S_{\alpha}=\{z\in \mathbb{C}: \Re z>0,|\Im z|\le (\Re z)\tan \alpha \} 
=\{re^{i\theta } : r>0,|\theta |\le \alpha \}. \] 
For $A\in \mathbb{M}_n$, the Cartesian (Toeptliz) decomposition 
is given as 
$A=\Re A+ i \cdot \Im A$, where 
\[ \text{$\Re A=\frac{1}{2}(A+A^*)$ ~and~ $\Im A=\frac{1}{2i}(A-A^*)$.} \] 
We know from the definition that if $W(A)\subseteq S_0$, 
then $A$ is positive definite.  
Moreover, 
it is easy to verify that  if $W(A)\subseteq S_{\alpha}$ for some 
$\alpha \in [0, {\pi}/{2})$, 
then $\Re (A)$ is positive definite. 
Such class of matrices whose numerical ranges are contained in a sector 
is called the sector matrices class. 
Clearly, the concept of sector matrices is an extension  of positive definite matrices. 
Over the past few years, 
various studies on sector matrices have been obtained in the literature; 
see, e.g., \cite{Choi19, Jiang19, Kua17, Lin15, YLC19, Zhang15}.

\medskip 

Before starting our results, 
we now summarise the following lemmas. 

\begin{lemma}  \cite{Lin15} \label{lem22}
Let $0\le \alpha < {\pi}/{2}$ and 
$A\in \mathbb{M}_n$ with $W(A)\subseteq S_{\alpha}$. Then 
\[ |\det A| \le (\sec \alpha)^n \det (\Re A). \]
\end{lemma}

\begin{lemma}  {\cite[p. 510]{HJ13}}  \label{lem23}
Let $X$ be an $n$-square complex matrix. Then 
\[ \lambda_i(\Re X) \le \sigma_i(X),\quad i=1,2,\ldots ,n. \]
Moreover, if $\Re X$ is positive definite, then 
\[ \det \Re X + |\det \Im X| \le |\det X|. \]
\end{lemma}

The following lemma is called the Fischer inequality, 
which gives an upper bound for the determinant of a positive
semidefinite block matrix in terms of the determinants of its principal diagonal
blocks. In particular, when all blocks have order $1\times 1$,  
this inequality is also known as the Hadamard inequality; 
see, e.g., \cite[p. 506]{HJ13} and \cite[p. 217]{Zhang11}. 

\begin{lemma} \label{lemfis}
Let $H=[H_{i,j}]_{i,j=1}^n \in \mathbb{M}_n(\mathbb{M}_k)$ be positive semidefinite. Then 
\[  \det H \le \prod_{i=1}^n \det H_{i,i}. \]
\end{lemma}

\begin{lemma}    \label{prop25}
If $H\in \mathbb{M}_n(\mathbb{M}_k)$ satisfies $W(H)\!\subseteq S_{\alpha}$, then 
$W(\tr_1 H) \!\subseteq S_{\alpha}$ and $W(\tr_2 H)\! \subseteq S_{\alpha}$, i.e.,  
if $H$ is a sector matrix with angle $\alpha \in [0, \pi /2)$, then 
so are $\tr_1 H$ and $\tr_2 H$. 
\end{lemma}

We remark that this lemma was partially proved in 
\cite[Proposition 3.2]{Kua17} for the case $\tr_2 H$. 
Motivated by \cite{Kua17}, 
we here include a detailed proof for the remaining case 
$\tr_1 H$. 

\medskip 

\begin{proof}
Consider the Cartesian decomposition $H=\Re H +i\cdot \Im H$, then 
\[ \tr_1 H = \tr_1 (\Re H) +i\cdot \tr_1 (\Im H). \] 
For every $x\in \mathbb{C}^k$ with $x^*x=1$, as $\Re H$ is positive definite, we get 
\[ \Re \bigl( x^*(\tr_1 H)x\bigr)= x^* \bigl( \Re(\tr_1 H) \bigr)x 
= x^* \bigl( \tr_1 (\Re H) \bigr)x >0. \]
On the other hand, by a direct computation, 
\[ \frac{\left|\Im \bigl( x^*(\tr_1 H)x \bigr)\right|}{\Re \bigl( x^*(\tr_1 H)x \bigr)} 
= \frac{\left| x^*(\tr_1 (\Im H))x \right|}{ x^*(\tr_1 (\Re H))x }
= \frac{\left| \langle  xx^*, \tr_1 (\Im H) \rangle \right|}{ 
\langle xx^*, \tr_1 (\Re H) \rangle } .    \]
Note that $I_n\otimes (xx^*)$ is positive semidefinite.  
We consider the spectral decomposition 
\[ I_n\otimes (xx^*) = 
\sum_{i=1}^{nk}\lambda_iu_iu_i^*,\]
 where $\lambda_i \ge 0$ and 
$u_i$ are unit vectors in $ \mathbb{C}^{nk}$. 
By the definition  in (\ref{eqdef}), 
it follows that 
\begin{equation*}
 \begin{aligned} 
\frac{\left| \langle  xx^*, \tr_1 (\Im H) \rangle \right|}{ 
\langle xx^*, \tr_1 (\Re H) \rangle }  
& =\frac{\left| \langle  I_n\otimes (xx^*), \Im H \rangle \right|}{ 
\langle I_n \otimes (xx^*), \Re H \rangle }  =
 \frac{\left| \sum_{i=1}^{nk} \lambda_i \langle  u_iu_i^*, \Im H \rangle \right|}{ 
\sum_{i=1}^{nk} \lambda_i \langle u_iu_i^*, \Re H \rangle } \\
&\le \frac{ \sum_{i=1}^{nk} \lambda_i \left| u_i^*(\Im H)u_i  \right|}{ 
\sum_{i=1}^{nk} \lambda_i u_i^*(\Re H )u_i }  \le 
\max_{1\le i\le nk} \frac{\left| u_i^*(\Im H)u_i  \right|}{u_i^*(\Re H )u_i} 
 = \max_{1\le i\le nk} \frac{\left| \Im (u_i^* Hu_i)  \right|}{ 
\Re (u_i^*Hu_i) }.
\end{aligned}
\end{equation*}
This completes the proof.  
\end{proof}

{\bf Remark.} 
Based on the second equivalent definition (\ref{eqdef2}), 
one could also give other ways to prove  Lemma  \ref{prop25}. 
We leave the details for the interested reader.

\section{Extensions on  Fiedler--Markham's inequality}
\label{sec3}

Let ${ H}=[H_{i,j}]_{i,j=1}^n \in \mathbb{M}_n(\mathbb{M}_k)$ be positive semidefinite.  
Recall that  both $\tr_1 H$ and $\tr_2 H$ 
are positive semidefinite; see, e.g., \cite{Zha12}.  
In 1994, 
Fiedler and Markham \cite[Corollary 1]{FM94} 
  proved a celebrated determinant inequality 
involving the second partial trace. 

\begin{theorem} \cite{FM94}  \label{thmfm}
Let ${ H}=[H_{i,j}]_{i,j=1}^n \in \mathbb{M}_n(\mathbb{M}_k)$ be
 positive semidefinite. Then 
\begin{equation*} \label{eqfm}
\left(\frac{\det \bigl( \tr_2 H\bigr)}{k }\right)^k \ge \det { H} .  
\end{equation*}
\end{theorem}

In 2016, Lin \cite{Lin16} revisited this inequality using 
some terminology from quantum information theory,  and 
 gave an alternative proof of  Theorem \ref{thmfm} 
 by applying an important identity 
 connecting $\tr_2 H$ and $H$. 
 Moreover, a natural question is that whether an analogous 
 result corresponding to the Fiedler--Markham inequality 
 holds for $\tr_1 H$.  
 Lin \cite{Lin16} answered this question and proved the following 
 counterpart.

\begin{theorem} \cite{Lin16}  \label{thmlin2016}
Let ${ H}=[H_{ij}]_{i,j=1}^n \in \mathbb{M}_n(\mathbb{M}_k)$ be
 positive semidefinite. Then 
\begin{equation*} \label{eq6lin}
 \left( \frac{\det (\tr_1 H)}{n}\right)^n \ge \det H. 
\end{equation*}
\end{theorem}

It is clear that in the proof of both Theorem \ref{thmfm} and 
Theorem \ref{thmlin2016}, 
Fiedler and Markham, and Lin used the superadditivity of determinant functional, which states that 
\[ \det \left(\sum_{i=1}^n H_{i,i} \right) \ge \sum_{i=1}^n \det H_{i,i}
\ge n \left( \prod_{i=1}^n \det H_{i,i} \right)^{1/n}. \] 
This inequality can be improved by the Fan-Ky determinant 
inequality (see  \cite[p. 488]{HJ13}), 
i.e., the log-concavity of the determinant over the cone of positive semidefinite matrices: 
\begin{equation} \label{eqfk} 
 \det \left( \frac{1}{n} \sum_{i=1}^n H_{i,i}\right) 
\ge \left( \prod_{i=1}^n \det H_{i,i}\right)^{1/n}. 
\end{equation}
In addition, we mention here that a careful examination of the new proof 
of Theorem \ref{thmfm} in  \cite{Lin16} can also reveal this improvement. 
This improvement was also pointed out in \cite{Choi17,LinZhang}.  
Next, we state the strong version of Theorem \ref{thmfm} and 
Theorem \ref{thmlin2016}.

\begin{theorem} \label{thmstrong}
Let ${ H}=[H_{ij}]_{i,j=1}^n \in \mathbb{M}_n(\mathbb{M}_k)$ be
 positive semidefinite. Then 
\begin{equation*}  \label{eqfmtr2}
\left(\frac{\det \bigl( \tr_2 H\bigr)}{k^n}\right)^k \ge \det { H} , 
\end{equation*}
and 
\begin{equation*}  \label{eqfmtr1}
 \left( \frac{\det (\tr_1 H)}{n^k}\right)^n \ge \det H. 
\end{equation*}
\end{theorem}

We  observe in Theorem \ref{thmstrong} that the second inequality   
seems easier to prove than the first inequality 
because it is more convenient to build inequalities on  $\tr_1 H
=\sum_{i=1}^n H_{i,i}$. 
In \cite{Li20}, the authors 
 showed that both inequalities 
can be deduced mutually.

In 2018, Kuai \cite{Kua17} (or see \cite{YLC19})
further extended Theorem \ref{thmstrong} to 
 sector matrices and 
showed that 
if $0\le \alpha < {\pi}/{2}$ and 
$H\in \mathbb{M}_n(\mathbb{M}_k)$ satisfies $W(H)\subseteq S_{\alpha}$, then 
\begin{equation} \label{eqkuaitr2}
  \left| \frac{\det (\tr_2 H)}{k^n} \right|^k \ge (\cos \alpha)^{nk} |\det H|,   
\end{equation}
and 
\begin{equation} \label{eqkuaitr1}
  \left| \frac{\det (\tr_1 H)}{n}\right|^n 
  \ge (\cos \alpha)^{(3n-2)k} |\det H|.  
\end{equation}

Our first goal in this section is to 
improve  Kuai's result (\ref{eqkuaitr1}). 
The key step in our improvement is the following identity 
 connecting $\tr_1(H)$ and $H$, 
which has been applied to quantum information theory, 
such as the sub-additivity of $q$-entropies. 
This identity can be found in 
 \cite[eq.(26)]{JR10} or \cite[Lemma 2]{Bes13}.

\begin{lemma} \label{lem24}
Let $X$ and $Y$ be generalized Pauli matrices on $\mathbb{C}^n$; 
these operators act as $Xe_j=e_{j+1}$ and $Ye_j=e^{2\pi j\sqrt{-1}/n}e_j$, 
where $e_j$ is the $j$-th column of the identity matrix $I_n$ and $e_{n+1}=e_1$. Then 
\begin{equation*} \label{eqhd}
\frac{1}{n}\sum_{l,j=1}^n 
(X^lY^j\otimes I_k)H (X^lY^j\otimes I_k)^*=I_n\otimes (\tr_1 H). 
\end{equation*}
\end{lemma}

{\bf Remark.} 
The identity in this lemma  can yield 
an alternative proof of Lemma \ref{prop25}. 
Moreover, the analogous identity 
for $\tr_2 H$ can be seen in  \cite{{JR10}} or \cite[eq.(14)]{Ras2012}.

\medskip 

Now, we are ready to present an improvement on  inequality (\ref{eqkuaitr1}). 

\begin{theorem} \label{thm21}
Let $0\le \alpha < {\pi}/{2}$ and $H\in \mathbb{M}_n(\mathbb{M}_k)$ be such that $W(H)\subseteq S_{\alpha}$. Then 
\begin{equation*}
 \left| \frac{\det (\tr_1 H)}{n^k}\right|^n 
 \ge  (\cos \alpha)^{nk} |\det H|.  
\end{equation*}
\end{theorem}

\begin{proof}
Note that 
both $X$ and $Y$ in Lemma \ref{lem24} are unitary, 
so are $X^lY^j\otimes I_k$ for all $l,j$. 
Moreover, we have $\Re (UHU^*)=U(\Re H)U^*$ for every unitary $U$.  
Thus, 
\begin{eqnarray}
|\det H|  \!\!\!\!\!\! &= &  \!\!\!\!\!\!\!\! \notag  \prod_{l,j=1}^n 
\left| \det (X^lY^j\otimes I_k)H (X^lY^j\otimes I_k)^*\right|^{1/n^2} \\
&  \overset{\text{Lemma \ref{lem22}}}{\le}& \!\!\!\!\!\!\!\! \notag 
 (\sec \alpha )^{nk}  
\prod_{l,j=1}^n \left( \det (X^lY^j\otimes I_k)(\Re H) (X^lY^j\otimes I_k)^*\right)^{1/n^2} \\
& \overset{\text{Fan-Ky ineq.(\ref{eqfk})}}{\le} & \!\!\!\!\!\!\!  \notag
 (\sec \alpha )^{nk}  \det \Bigg( \frac{1}{n^2}
\sum_{l,j=1}^n   (X^lY^j\otimes I_k)(\Re H) (X^lY^j\otimes I_k)^*\Bigg)\\
& \overset{\text{Lemma \ref{lem24}}}{=} &  \!\!\!\!\!\!\!\!   \notag
(\sec \alpha )^{nk} \det \left( \frac{1}{n}\Bigl( I_n\otimes \tr_1(\Re H)\Bigr) \right) \\
&  = &  \!\!\!\!\!\!\!\!\! \frac{(\sec \alpha )^{nk}}{n^{nk}}
 \det \Bigl( I_n\otimes \tr_1(\Re H)\Bigr). \label{eqb1}
 \end{eqnarray} 
Clearly,  we have $\tr_1 (\Re H)=\Re (\tr_1 H)$. 
For $X  \in \mathbb{M}_n$ and $Y \in \mathbb{M}_k$, 
it is well-known that
 $\det (X \otimes Y) = (\det X)^k (\det Y)^n$; see, e.g., \cite[Chapter 2]{Zhan13}. 
It follows that 
\begin{equation*} 
\det \Bigl( I_n\otimes \tr_1(\Re H)\Bigr) 
= (\det I_n)^k \bigl( \det (\tr_1 \Re H) \bigr)^n  
=  \bigl( \det \Re (\tr_1 H) \bigr)^n.
\end{equation*}
By Proposition \ref{prop25}, we have $W( \tr_1 H)\subseteq S_{\alpha}$, 
which implies that $\Re (\tr_1 H)$ is positive definite. 
Therefore, by Lemma \ref{lem23}, we get 
\begin{equation*}
  \bigl( \det \Re (\tr_1 H) \bigr)^n 
\le    \bigl(|\det (\tr_1 H)|-|\det \Im (\tr_1 H)|\bigr)^n 
\le  |\det (\tr_1 H)|^n,   
\end{equation*}
which together with (\ref{eqb1})  yields the desired result. 
\end{proof}

{\bf Remark.} 
By applying the techniques from \cite{Li20}, 
we  know that  Kuai's inequality (\ref{eqkuaitr2})
can also  be deduced from the inequality in Theorem \ref{thm21}
 and vice versa.

\medskip 

In the sequel, we shall focus our attention on 
some recent results which are similar with 
the Fiedler--Markham inequality. 
Let ${ H}=[H_{i,j}]_{i,j=1}^n \in \mathbb{M}_n(\mathbb{M}_k)$ be 
a block matrix with $H_{i,j}=[ h_{l,m}^{i,j}]_{l,m=1}^k$. 
We define an $n\times n$ matrix $G_{l,m}$ as  below. 
\[ G_{l,m}:=\bigl[ h_{l,m}^{i,j}\bigr]_{i,j=1}^n\in \mathbb{M}_n .\] 
A direct computation yields 
\[ \tr_1 H =\sum_{i=1}^n H_{i,i}=\sum_{i=1}^n \bigl[ h_{l,m}^{i,i}\bigr]_{l,m=1}^k
=\left[ \begin{matrix} \sum\limits_{i=1}^n h_{l,m}^{i,i} \end{matrix}\right]_{l,m=1}^k 
=\bigl[ \tr \,G_{l,m}\bigr]_{l,m=1}^k. \] 
For notational convenience, we denote 
\[ \widetilde{H}= \bigl[ G_{l,m} \bigr]_{l,m=1}^k \in \mathbb{M}_k(\mathbb{M}_n).\]  
We can see that $\widetilde{H}$ is obtained 
from $H$ by rearranging the entries in an appropriate order. 
The above observation yields 
$\tr_1 H = \tr_2 \widetilde{H}$. 
Moreover, it is not hard to check that 
$\widetilde{H}$ and $H$ are unitarily similar; 
see, e.g., \cite[Theorem 7]{Choi17} or \cite[Theorem 4]{Li20}. 
Motivated by these relations, 
Choi \cite{Choi17}  introduced recently  the definition of partial determinants 
corresponding to partial traces. For $H=[H_{i,j}]_{i,j=1}^n\in \mathbb{M}_n(\mathbb{M}_k)$, 
the partial determinants are defined as 
  \[  \mathrm{det}_1H :=\bigl[ \det G_{l,m}\bigr]_{l,m=1}^k, \]
  and 
\[ \mathrm{det}_2 H :=\bigl[ \det H_{i,j}\bigr]_{i,j=1}^n.  \] 
To some extent, 
the partial determinants share some common properties relative to partial traces. 
For instance, 
it is easy to see that if ${ H} \in \mathbb{M}_n(\mathbb{M}_k)$ is positive semidefinite, then both $\mathrm{det}_1 H$ and $\mathrm{det}_2 H$ are positive semidefinite; see, e.g. \cite[p. 221]{Zhang11}.
Moreover, it was proved in \cite{Choi17} that 
\[  \mathrm{det} (\tr_1 H) \ge \tr (\mathrm{det}_2 H) ,\] 
and 
\[ \mathrm{det} (\tr_2 H) \ge \tr (\mathrm{det}_1 H).  \]
Additionally, Choi \cite{Choi17} proved 
two analogues of Theorem \ref{thmfm} and  Theorem \ref{thmlin2016} for partial determinants. 

\begin{theorem} \cite{Choi17}  \label{thmchoi11}
Let ${ H} \in \mathbb{M}_n(\mathbb{M}_k)$ be positive semidefinite. 
Then 
\begin{equation*} \label{eqchoi1}
 \left( \frac{\tr (\mathrm{det}_1 H)}{k}\right)^k \ge \det H,
\end{equation*}
and 
\begin{equation*} \label{eqchoi2}
 \left( \frac{\tr (\mathrm{det}_2 H)}{n}\right)^n \ge \det H. 
\end{equation*}
\end{theorem}

Next, we will  
extend Theorem \ref{thmchoi11} to sector matrices. 
We write $|A|$ for the nonnegative matrix 
whose entries are the absolute of the entries of $A$. 
This notation is only used in the following theorem.

\begin{theorem}\label{thm26}
Let $0\le \alpha <{\pi}/{2}$ and 
$H\in \mathbb{M}_n(\mathbb{M}_k)$ be such that $W(H)\subseteq S_{\alpha}$. Then 
\begin{equation*} \label{eqext1}
  \left( \frac{\tr |\mathrm{det}_1 H|}{k}\right)^k 
  \ge (\cos \alpha )^{nk}  |\det H|, 
\end{equation*}
and 
\begin{equation*} \label{eqext2}
   \left( \frac{\tr |\mathrm{det}_2 H|}{n}\right)^n 
   \ge (\cos \alpha )^{nk} |\det H| . 
\end{equation*}
\end{theorem}

\begin{proof}
First of all, we shall prove the second inequality. 
We observe that 
$\Re H_{1,1}$,  $\ldots ,\Re H_{n,n}$ 
are the diagonal block matrices of $\Re H$. 
By Lemma \ref{lem22} and Lemma \ref{lemfis}, we obtain  
\begin{align*}
|\det H| &\le (\sec \alpha)^{nk} \det (\Re H) 
 \le (\sec \alpha)^{nk} \prod_{i=1}^n \det (\Re H_{i,i})  \\
&\le (\sec \alpha)^{nk} \prod_{i=1}^n |\det H_{i,i}| 
\le (\sec \alpha)^{nk} \left( \frac{1}{n} \sum_{i=1}^n |\det H_{i,i}|\right)^n,  
\end{align*}
where the third inequality follows from Lemma \ref{lem23} 
and the last one follows from the arithmetic mean-geometric mean inequality. 

We now prove the first desired inequality
 by employing the relations between $\mathrm{det}_1$ and $\mathrm{det}_2$. 
Recall that $\widetilde{H}=[G_{l,m} ]_{l,m=1}^k\in \mathbb{M}_k(\mathbb{M}_n)$ and  $\mathrm{det}_1 H=\mathrm{det}_2 \widetilde{H}$. 
Since $\widetilde{H}$ and $H$ are unitarily similar,  
we can get $\det \widetilde{H}=\det H$ and $W(\widetilde{H})\subseteq S_{\alpha}$. 
Moreover, $\widetilde{H}$   is also positive semidefinite. 
By applying the second inequality to $\widetilde{H}$, 
we  get 
\[  \left( \frac{\tr |\mathrm{det}_1 H|}{k}\right)^k
= \left( \frac{\tr |\mathrm{det}_2 \widetilde{H}|}{k}\right)^k 
\ge (\cos \alpha )^{kn}  |\det \widetilde{H}| 
= (\cos \alpha )^{kn} |\det H|.  \]
This completes the proof. 
\end{proof}

\section{Extensions on Ando's inequality} 
\label{sec4}

To make our statements more transparent and compatible 
with the previous works in the literature. In this section, 
we assume that 
$A$ is an $m\times m$ block matrix with each block being 
an $n \times n$ matrix.  
Let ${ A}=[A_{i,j}]_{i,j=1}^m \in \mathbb{M}_m(\mathbb{M}_n)$ be positive semidefinite.  
We know that  both $\tr_1 A$ and $\tr_2 A$ 
are positive semidefinite; see, e.g., \cite[p. 237]{Zhang11} 
and \cite[Theorem 2.1]{Zha12}. 
To some degree, 
these two partial traces are closely related and 
mutually affect each other.  
We write $\lVert A \lVert_q = 
\left( \sum_{i} \sigma_i (A)^q \right)^{1/q}$ 
for the Schatten $q$-norm of $A$. 
In 2007,  Audenaert \cite{Aud07} proved the 
following norm inequality, 
\begin{equation} \label{eqaud}
\tr \,A + \lVert A \lVert_q \ge \lVert \tr_1 A \rVert_q + \lVert \tr_2 A \rVert_q .
\end{equation}
A straightforward argument exploiting   Audenaert's result 
leads to a proof of the subadditivity of $q$-entropies (Tsallis entropies)  
 for finite-dimensional bipartite quantum states; see \cite{Aud07,Bes13} and references therein. 
In 2014, Ando  \cite{Ando14}  
(or see \cite[Proposition 2.2]{Lin16b} for an alternative proof)
established  the following remarkable inequality 
in the sense of  the L\"{o}wner ordering. 

\begin{theorem} \cite{Ando14,Lin16b} \label{thmando}
Let $A \in \mathbb{M}_m(\mathbb{M}_n)$ be positive semidefinite. Then  
\begin{equation*} 
(\tr A)I_{mn}+  A \ge  
I_m\otimes (\mathrm{tr}_1 A)  + (\mathrm{tr}_2 A) \otimes I_n.
\end{equation*}
\end{theorem}

Ando's result  reveals closely the  interplay between the first and second partial trace. 
Equivalently, this inequality can be rewritten as 
\begin{equation} \label{eqando}
(\tr A)I_{mn} - (\mathrm{tr}_2 A) \otimes I_n   \ge  
I_m\otimes (\mathrm{tr}_1 A)  -A .
\end{equation}
We observe that the positivity of  $A$, 
together with the identity 
$\tr \, A =\sum_{i=1}^m \tr A_{i,i} =
 \tr (\tr_2 A)$, 
 leads to   
$ (\tr A)I_m \ge \lambda_{\max}(\tr_2 A)I_m \ge  \tr_2 A$,
 which guarantees that in (\ref{eqando}) the left hand side 
$(\tr A)I_{mn} - (\mathrm{tr}_2 A) \otimes I_n $  is positive semidefinite. 
However, the two matrices of the 
right hand side in (\ref{eqando}) might be incomparable. 
For instance, the matrix $A$ in (\ref{eqeq1}) gives an example. 
Motivated by this observation, 
Li, Liu and Huang \cite{HuangLi20}  presented 
a further generalization. 

\begin{theorem} \cite{HuangLi20}  \label{thmllh}
Let $A \in \mathbb{M}_m(\mathbb{M}_n)$ be positive semidefinite. Then  
\begin{equation*} \label{eqhuang}
(\tr A)I_{mn} -(\tr_2A)\otimes I_n \ge A - I_m \otimes (\tr_1 A),  
\end{equation*}
and 
\begin{equation*} \label{eqhuang2}
(\tr A)I_{mn}  + (\tr_2A)\otimes I_n \ge  A +  I_m \otimes (\tr_1 A).  
\end{equation*}
\end{theorem}

A  map (not necessarily linear) $\Phi: \mathbb{M}_n\to \mathbb{M}_k$ is called positive 
if it maps positive semidefinite matrices 
to positive semidefinite matrices. 
A  map $\Phi: \mathbb{M}_n\to \mathbb{M}_k$  is said to be {\it $m$-positive} if 
for every $m\times m$ block matrix 
$[A_{i,j}]_{i,j=1}^m\in \mathbb{M}_m(\mathbb{M}_n)$,  
\begin{equation*}  \label{eq1}
[A_{i,j}]_{i,j=1}^m \ge 0 \Rightarrow [\Phi (A_{i,j})]_{i,j=1}^m\ge 0. 
\end{equation*}
Clearly, being $1$-positive is equivalent to
being positive. 
The   map $\Phi$ is said to be {\it completely positive} 
if it is $m$-positive for every  integer $m\ge 1$. 
It is well-known that both the trace map and  determinant map 
are  completely positive; see, e.g., \cite[p. 221, p. 237]{Zhang11} 
or \cite{Zha12}. 
On the other hand, a  map  $\Phi $  is said to be {\it $m$-copositive} if 
for every $[A_{i,j}]_{i,j=1}^m\in \mathbb{M}_m(\mathbb{M}_n)$, 
\begin{equation*}  \label{eq2}
[A_{i,j}]_{i,j=1}^m \ge 0 \Rightarrow [\Phi (A_{j,i})]_{i,j=1}^m\ge 0,  
\end{equation*}
and $\Phi$ is said to be {\it completely copositive} 
if it is $m$-copositive for every  integer $m\ge 1$. 
Furthermore, 
a map $\Phi$ is called {\it completely PPT} if it is both completely positive and completely copositive; 
see \cite{Lin14, FLTmia, Zhang19} for related topics. 

Both Theorem \ref{thmando} 
and Theorem \ref{thmllh}
illustrated the implicit interaction and connection between the first trace and second trace.  
The proof of Theorem \ref{thmando} depends mainly on 
the 2-copositivity of $\Psi (X)=(\tr X)I-X$; see e.g., \cite{Ando14} and \cite{Lin16b} for more details. 
Correspondingly, the proof  of Theorem \ref{thmllh}  relies similarly  on 
the 2-copositivity of $\Phi (X)=(\tr X)I+X$; 
see \cite{HuangLi20}. 
For more application of these two maps, 
we refer readers to papers \cite{Lin14, Li20laa}. 

In this section, 
we give a unified treatment of both Theorem \ref{thmando} 
and Theorem \ref{thmllh}.  
Our treatment is more concise than the original proof.  
We need to use a recent result of Choi \cite{Choi17,Choi18}, 
which  investigates more relations between 
the partial traces and 
the partial transpose.

\begin{lemma}  \cite{Choi17,Choi18}  \label{thmchoi}
Let $A\in \mathbb{M}_m(\mathbb{M}_n)$ be positive semidefinite. Then  
\begin{equation*} \label{eqchoi2}
 \text{$(\tr_2 A^{\tau})\otimes I_n\ge \pm A^{\tau}$,} 
\end{equation*}
and
\begin{equation*} \label{eqchoi1}
I_m\otimes \tr_1 A^{\tau} \ge \pm A^{\tau}.
\end{equation*}
\end{lemma}

Now, we present a unified treatment of 
  Theorems \ref{thmando} 
and  \ref{thmllh} as well.  
\medskip 

\noindent 
{\bf New proof of Theorem \ref{thmando}.} 
 We define the map $\Phi : \mathbb{M}_m(\mathbb{M}_n) 
\to \mathbb{M}_m(\mathbb{M}_n)$ as 
\[   \Phi_{2}^{-} (X) :=(\tr_2 X^{\tau})\otimes I_n- X^{\tau}.\]   
On the other hand, we define  
\[  \Phi_1^- (X):=I_m \otimes \tr_1 X^{\tau} -X^{\tau}.\]   
Lemma \ref{thmchoi} 
implies that 
both $\Phi_2^-$ and $\Phi_1^-$  are positive linear maps 
on $\mathbb{M}_m(\mathbb{M}_n)$. 
Let $A$ be a positive semidefinite block matrix.  
Thus, we have 
\[  \Phi_2^-(A)=(\tr_2 A^{\tau})\otimes I_n- A^{\tau} \ge 0.\]   
Acting the map $\Phi_1^-$ to the matrix $\Phi_2^- (A)$,  we can obtain  
\begin{equation} \label{eqkey}
\Phi_1^- \bigl( \Phi_2^- (A) \bigr)= I_m\otimes \tr_1 {\Phi_2^- (A)}^{\tau} 
- {\Phi_2^- (A)}^{\tau}\ge 0.  
\end{equation}
By a directed computation, we can get 
$ {\Phi_2^-(A)}^{\tau} = (\tr_2 A)\otimes I_n - A$   
and 
\[  \tr_1 {\Phi_2^- (A)}^{\tau} = \tr_1 \bigl( (\tr_2 A )\otimes I_n -A\bigr) = 
\sum_{i=1}^m (\tr A_{i,i})I_n - \tr_1 A =(\tr A)I_n -\tr_1 A.   \]
Therefore,  inequality (\ref{eqkey}) yields 
the desired result in Theorem \ref{thmando}.  $\blacksquare$

\medskip 
\noindent 
{\bf Remarks.}~
In the above proof, we can see that Theorem \ref{thmando} 
is just a direct consequence of Lemma \ref{thmchoi}. 
To our surprise, 
Theorem \ref{thmando}  can also be proved  by using 
the positivity of $\Phi_1^-$ first, 
and then applying the positivity of $\Phi_2^-$ later.   
More precisely,  we first derive   
${\Phi_1^- (A)}\ge 0$, 
and then we have $\Phi_2^- (\Phi_1^- (A))\ge 0$. 
Upon simplification,  
one can immediately get  Theorem \ref{thmando}  again.  
We summarize this observation as the following proposition. 

\begin{proposition}
For every $X\in \mathbb{M}_m(\mathbb{M}_n)$,  we have 
$\Phi_1^- (\Phi_2^- (X)) = \Phi_2^- (\Phi_1^- (X))$. 
\end{proposition}

Correspondingly, we can present an alternative proof of 
Theorem \ref{thmllh} similarly. 

\medskip 

\noindent 
{\bf New proof of Theorem \ref{thmllh}.} 
We  define the maps $\Phi_2^+$ and $\Phi_1^+$ 
 on $\mathbb{M}_m(\mathbb{M}_n)$  as 
\[ 
 \Phi_2^+ (X) :=(\tr_2 X^{\tau})\otimes I_n + X^{\tau}, 
\]
and 
\[  \Phi_1^+ (X):=I_m \otimes \tr_1 X^{\tau}  + X^{\tau}.   \]
We can see from Lemma \ref{thmchoi} that 
both $\Phi_2^+$ and $\Phi_1^+$ are  positive linear maps. 
Similar to the lines of the previous proof, we  get 
$ \Phi_1^{-} ( \Phi_2^+ (A)) =
\Phi_2^{+} ( \Phi_1^{-} (A)) \ge 0 $, which leads to 
\[  (\tr A)I_{mn} -(\tr_2A)\otimes I_n \ge A - I_m \otimes (\tr_1 A).  \]
Moreover, we have $ \Phi_1^+ (\Phi_2^- (A)) = 
\Phi_2^{-} (\Phi_1^{+} (A)) \ge 0$. 
It follows that 
\[  (\tr A)I_{mn} + (\tr_2A)\otimes I_n \ge A + I_m \otimes (\tr_1 A).  \]
We mention that the positivity of $\Phi_1^+ (\Phi_2^+ (A))$ yields 
a trivial result. 
$\blacksquare$

\medskip 

In the remaining of this section, we shall pay attention to 
determinant inequalities of sector matrices 
involving partial traces. 
Motivated by Audenaert's result (\ref{eqaud}), 
Lin \cite{Lin16} recently  obtained a  determinantal inequality for partial traces, which states that 
if $A\in \mathbb{M}_m 
(\mathbb{M}_n)$ is positive semidefinite, then 
\begin{equation} \label{eqlin}
(\tr A)^{mn} +\det A   \ge   \det (\tr_1 A)^m +\det (\tr_2 A)^n. 
\end{equation}
We remark here that 
Fu, Lau and Tam \cite[Corollary 2.2]{FLT20} recently improved (\ref{eqlin})  
when $A$ is a density matrix, i.e., a positive semidefinite matrix with 
trace equal to $1$. 

The key step in the proof of (\ref{eqlin}) attributes to 
Theorem \ref{thmando} together with 
 the following interesting lemma. 
It is worth noting that Lemma \ref{lem31} is graceful and useful in deriving matrix inequalities; 
see, e.g., \cite{LF2021,LP2021,Lin14b} for  applications on 
Oppenheim type inequalities.

 \begin{lemma}  \cite{Lin16b} \label{lem31}  
Let $X, Y, W $ and $Z$ be positive semidefinite matrices of the same order.  
If  $X\ge W,X\ge Z$ and $X+Y\ge W+Z$, then
 	\[ \label{lin} \det X+\det Y\ge \det W+\det Z. \]
\end{lemma}

\noindent 
{\bf Remark.}~
We observe that  Lemma \ref{lem31} implies the determinant inequality:  
\[ \det (A+B+C) +\det C \ge \det (A+C) +\det (B+C), \]
where $A,B$ and $C$ are positive semidefinite matrices.  

\vspace{0.4cm}

With the help of Lemma \ref{lem31}, 
we can easily  present two analogues of  (\ref{eqlin}). 

\begin{proposition} \label{prop42}
 Let $A\in \mathbb{M}_m(\mathbb{M}_n)$ be positive semidefinite. Then
\begin{equation*}  \label{eq42}
 	 (\tr A)^{mn} +\det (\tr_1 A)^m \ge \det A +\det (\tr_2 A)^n, 
\end{equation*}
and  
\begin{equation*}  \label{eq43}
	(\tr A)^{mn} + \det(\tr_2 A)^n \ge 
   \det A + \det(\tr_1 A)^m .
\end{equation*}
\end{proposition} 
 
\begin{proof}
We prove the first inequality only, since the second one
 can be proved in exactly the same way.  
Let $X=(\tr A)I_{mn}, Y=I_m\otimes (\tr_1 A),W=A$ and $Z=(\tr_2 A)\otimes I_n$. It is easy to see that 
\begin{equation*} \label{eqeq} 
(\tr A)I_m =\sum_{i=1}^m (\tr A_{i,i})I_m =
\bigl( \tr (\tr_2 A)\bigr)I_m \ge \lambda_{\max}(\tr_2 A) I_m 
\ge \tr_2 A, 
\end{equation*}
which implies that $X\ge Z\ge 0$, and clearly $X\ge W\ge 0$. 
Moreover, Theorem \ref{thmllh} says that  
$X+Y\ge W+Z$. That is, all conditions in Lemma \ref{lem31} are satisfied. Therefore, 
\begin{align*}
(\tr A)^{mn} + \det \bigl( I_m\otimes (\tr_1 A) \bigr)
\ge  \det A  +\det \bigl( (\tr_2 A) \otimes I_n \bigr).
\end{align*}
It is well-known \cite[p. 37]{Zhan13} that for every $X\in \mathbb{M}_m$ and $Y\in \mathbb{M}_n$, 
\[ \det (X\otimes Y) =(\det X)^n(\det Y)^m.  \]
Thus, we complete the proof of the required result. 
\end{proof}

We next give an improvement on Proposition \ref{prop42}. 

\begin{theorem} \label{prop44}
 Let $A\in \mathbb{M}_m(\mathbb{M}_n)$ be positive semidefinite. Then
\[  (\tr A)^{mn} +  \det (\tr_1 A)^m \ge  m^{nm} \bigl( \det A +\det (\tr_2 A)^n\bigr), \]
and 
\[ 	(\tr A )^{mn} + \det(\tr_2 A)^n \ge 
 n^{mn} \bigl(  \det A + \det(\tr_1 A)^m \bigr). \]
\end{theorem}

\begin{proof}
We only prove the second inequality.  
Invoking Theorem \ref{thmstrong}, we get 
\begin{equation*} 
 \left( \frac{\det (\tr_2 A)}{n^m} \right)^n \ge \det A.
\end{equation*}
Equivalently, we have $  \det (\tr_2 A)^n \ge n^{mn} \det A$. 
 It suffices to show that 
\[  (\tr A)^n \ge n^n \det (\tr_1 A). \]  
Note that 
\[ \tr A= \sum_{i=1}^m \tr (A_{i,i}) = \tr \left(\sum_{i=1}^m A_{i,i} \right) 
=\tr (\tr_1 A). \] 
We denote $X:=\tr_1A$, which 
is a positive semidefinite matrix of order $n$. 
So we need to prove that  $(\tr X)^n\ge n^n \det X$. 
This is equivalent to showing 
\[  \left( \sum_{i=1}^n \lambda_i(X) \right)^n \ge n^n \prod_{i=1}^n 
\lambda_i(X), \] 
which is a direct consequence of the AM-GM inequality. 
\end{proof}

Surprisingly, the proof of Theorem \ref{prop44} seems simpler than 
that of Proposition \ref{prop42} since 
it does not rely on Theorem \ref{thmllh} and Lemma \ref{lem31}. 
However, it allows us to provide
 a great improvement on Proposition \ref{prop42} whenever 
 $m,n$ are large integers.

In the sequel, we shall denote $|A|=(A^*A)^{1/2}$, 
which is called the modulus of $A$.  
We remark that this notation is different from that  in Theorem \ref{thm26}. 
Note that $|A|$ is positive semidefinite, and the eigenvalues of $|A|$ 
are called  the singular values of $A$. 
In 2019, Yang, Lu and Chen \cite{YLC19} 
extended (\ref{eqlin}) to sector matrices. 

\begin{equation*}
 (\tr |A|)^{mn} + \det |A| \ge (\cos \alpha)^{mn} |\det (\tr_1 A)|^m + 
(\cos \alpha)^{mn} |\det (\tr_2 A)|^n. 
\end{equation*}

Now, we are ready to present an extension on Theorem \ref{prop44}.

\begin{theorem} \label{thm36}
Let $A\in \mathbb{M}_m(\mathbb{M}_n)$ be such that $W(A)\subseteq S_{\alpha}$. Then 
\begin{equation*}
 (\tr |A|)^{mn} + 
\left| {\det (\tr_1 A)}\right|^m 
\ge 
(m \cos \alpha )^{mn} \bigl( \det |A| +  |\det (\tr_2 A)|^n\bigr),
\end{equation*}
and 
\begin{equation*}
 (\tr |A|)^{mn} + \left| {\det (\tr_2 A)}\right|^n \ge 
(n \cos \alpha )^{mn} 
\bigl( \det |A| + |\det (\tr_1 A)|^m \bigr).
\end{equation*}
\end{theorem}

\begin{proof} 
We only prove the first inequality. 
According to the definition of $S_{\alpha}$, 
if $W(A)\subseteq S_{\alpha}$, then $\Re A$ is positive definite and its trace is positive. 
By Lemma \ref{lem23}, we have 
\begin{equation*}\label{eq12} 
\tr |A| =\sum_{i=1}^{mn} \sigma_i(A) \ge \sum_{i=1}^{mn} \lambda_i(\Re A) =
\tr (\Re A) \ge 0. 
\end{equation*}
It is noteworthy by Lemma \ref{prop25} 
that $W(\tr_1 A) \subseteq S_{\alpha}$ and 
$W(\tr_2 A) \subseteq S_{\alpha}$. 
Clearly, we have
 $\Re (\tr_1 A) = \tr_1( \Re A)$ and $\Re( \tr_2  A) =  \tr_2( \Re A)$. 
By setting $X=\tr_1 A$ in Lemma \ref{lem23}, 
we get 
\[ |\det (\tr_1 A)| \ge \det \bigl( \Re (\tr_1 A) \bigr) 
=\det \bigl( \tr_1( \Re A) \bigr).  \]  
Note that $\Re A$ is positive semidefinite. 
By applying Theorem \ref{prop44}, we can obtain 
\begin{align*}
 (\tr |A|)^{mn} + 
\left| {\det (\tr_1 A)}\right|^m   
&\ge  (\tr \, \Re A)^{mn} + 
  \bigl( {\det \tr_1( \Re A)} \bigr)^m\\ 
& \ge m^{nm}\bigl( \det (\Re A) + \bigl(  \det  \Re( \tr_2  A) \bigr)^n \bigr) \\ 
&\ge (m\cos \alpha)^{mn} |\det A| + (m\cos \alpha)^{mn} |\det (\tr_2 A) |^n, 
\end{align*}
where the last inequality holds 
from Lemma \ref{lem23} by setting $X=A$ and $\tr_2 A$  
respectively. 
\end{proof}

\section{Trace inequalities for two by two  block matrices}

\label{sec5}

Positive semidefinite $2\times 2$ block matrices are extensively studied, 
such a partition yields a great deal of versatile and elegant matrix inequalities;  
see, e.g., \cite{Gumus18, KL17, Li20laa, FG2021} for details.  
Recently, Kittaneh and Lin \cite{KL17} (or see \cite{Lin14}) proved the following 
 trace inequalities. 
 
 \begin{theorem} \cite{KL17,Lin14} \label{thm51}
Let $\begin{bmatrix}A & B \\ B^* & C \end{bmatrix}\in \mathbb{M}_{2}(\mathbb{M}_{k})$ 
be positive semidefinite. Then 
\begin{equation*} \label{eqtr}
  \tr A \,\tr C -\tr B^*\,\tr B \ge \bigl| \tr AC -\tr B^*B \bigr|, 
\end{equation*}
and 
\begin{equation*} \label{eqtrr}
  \tr A \,\tr C +\tr B^*\,\tr B \ge \tr AC + \tr B^*B.  
\end{equation*}
\end{theorem}

In this section, we  present some  
inequalities related to trace for $2\times 2$ block matrices, 
which are slight extensions of the result of Kittaneh and Lin. 
We now need to introduce some  notations. 
Let $\otimes^r A:=A\otimes \cdots \otimes A$ be the $r$-fold tensor power of $A$.

\begin{theorem} \label{thm52}
Let $\begin{bmatrix}A & B \\ B^* & C \end{bmatrix}\in \mathbb{M}_{2}(\mathbb{M}_{k})$ 
be positive semidefinite. Then  for $r\in \mathbb{N}^*$, 
\[  (\mathrm{tr}A \,\mathrm{tr} C)^r - (\tr B^* \,\mathrm{tr} B)^{r} 
\ge \bigl| (\mathrm{tr} AC)^r - (\mathrm{tr} B^*B)^r\bigr| , \]
and 
\[   (\mathrm{tr}A \,\mathrm{tr} C)^r + (\tr B^* \,\mathrm{tr} B)^{r} 
\ge  (\mathrm{tr} AC)^r + (\mathrm{tr} B^*B)^r. \] 
\end{theorem}

\begin{proof}
Note that $\begin{bmatrix}\!\!\! \otimes^r A & \otimes^r B \\ \otimes^r B^* & \otimes^rC \end{bmatrix}$ 
is a principal submatrix of $\otimes^r \begin{bmatrix}A & B \\ B^* & C \end{bmatrix}$. 
Thus 
\[ \begin{bmatrix}\!\!\! \otimes^r A & \otimes^r B \\ \otimes^r B^* & \otimes^rC \end{bmatrix} \]
is again positive semidefinite. 
By applying Theorem \ref{thm51} to this block matrix, we get 
\[ 
\bigl| \tr (\otimes^r A )(\otimes^r C )-\tr (\otimes^r B^*)( \otimes^r B )\bigr| 
\le \tr \otimes^{r}\!\!A \,\tr \otimes^r\!\!C -\tr \otimes^r \!B^*\tr \otimes^r\!\! B, 
\]
and 
\[  
 \tr (\otimes^r A)( \otimes^r C )+\tr (\otimes^r B^*)( \otimes^r B )
\le \tr \otimes^{r}\!\!A \,\tr \otimes^r\!\!C  + \tr \otimes^r \!B^*\tr \otimes^r\!\! B.
\]
Invoking the well-known 
facts \cite[Chapter 2]{Zhan13}: $(\otimes^r X)( \otimes^rY) =\otimes^r(XY)$ and $\tr (\otimes^rX)=(\tr \,X)^r$, 
the desired inequalities  follow immediately. 
\end{proof}

\noindent 
 {\bf Remark.} Theorem \ref{thm52} was proved in the first version of our manuscript (announced on March 10, 2020, arXiv: 
 \href{https://arxiv.org/abs/2003.04520v1}{2003.04520v1}). We remark  that this result was recently and independently rediscovered  by Fu and Gumus in \cite{FG2021} using a quite different method.

\medskip 
Let $e_t(X)$ denote the $t$-th elementary symmetric function of 
the eigenvalues of the square matrix $X$.  
\[ e_t(X):=\sum_{1\le i_1< i_2 < \cdots < i_t\le n} 
\prod_{j=1}^t \lambda_{i_j}(X). \]  
In particular, we know that $e_1(X)=\tr(X)$. 
We can get the following theorem. 

\begin{theorem} \label{thm53}
Let $\begin{bmatrix}A & B \\ B^* & C \end{bmatrix}\in \mathbb{M}_{2}(\mathbb{M}_{k})$ 
be positive semidefinite. Then  for $t\in \{1,2,\ldots ,k\}$,  
\[  e_t(A)e_t(C) - e_t(B^*)e_t(B)\ge 
|e_t(AC) -e_t(B^*B)|, \]
and 
\[  e_t(A)e_t(C) + e_t(B^*)e_t(B) 
\ge e_t(AC) + e_t(B^*B). \] 
\end{theorem}

\begin{proof} 
The first inequality can be found in \cite[Corollary 2.7]{KL17}. 
We next give the outline of  the proof 
of the second one. 
Note that $\begin{bmatrix} \!\!\! \otimes^t A & \otimes^t B \\ \otimes^t B^* & \otimes^t C \end{bmatrix}$ 
is  positive semidefinite.  
By restricting this block matrix to 
the symmetric class of tensor product (see, e.g., \cite[pp. 16--20]{Bha97}), we know that 
\[ \begin{bmatrix} \!\!\! \wedge^t A & \wedge^t B \\ 
\wedge^t B^* & \wedge^t C \end{bmatrix}  \]
is still positive semidefinite.  Note that $e_t(X)=\tr (\wedge^t X)$ and 
$(\wedge^t X) (\wedge^t Y) = \wedge^t  (XY)$.  
Applying Theorem \ref{thm51}
 to this block matrix yields the required result. 
\end{proof}

Let $s_t(X)$ be the $t$-th complete symmetric polynomial 
of  eigenvalues of $X$, i.e.,
\[ s_t(X):=\sum_{1\le i_1\le i_2\le \cdots \le i_t\le n} 
\prod_{j=1}^t \lambda_{i_j}(X). \]  
 Clearly, we have $s_1(X)=\tr (X)$. 
 We can get the following slight extension similarly. 

\begin{theorem} \label{thm54}
Let $\begin{bmatrix}A & B \\ B^* & C \end{bmatrix}\in \mathbb{M}_{2}(\mathbb{M}_{k})$ 
be positive semidefinite.  Then  for $t\in \{1,2,\ldots ,k\}$,  
\[  s_t(A)s_t(C) - s_t(B^*)s_t(B)\ge 
|s_t(AC) -s_t(B^*B)|, \]
and 
\[  s_t(A)s_t(C) + s_t(B^*)s_t(B) 
\ge s_t(AC) + s_t(B^*B). \] 
\end{theorem}

\begin{proof}
Note that $\begin{bmatrix} \!\!\! \otimes^t A & \otimes^t B \\ \otimes^t B^* & \otimes^t C \end{bmatrix}$ 
is  positive semidefinite.  
By restricting this block matrix to 
the symmetric class of tensor product (see, e.g., \cite[pp. 16--20]{Bha97}), we know that 
\[ \begin{bmatrix} \!\!\! \vee^t A & \vee^t B \\ \vee^t B^* & \vee^t C \end{bmatrix}  \]
is still positive semidefinite. Similarly, we know that  $\tr (\vee^tX)=s_t(X)$ and 
 $(\vee^t X)( \vee^tY) =\vee^t(XY)$. 
Applying Theorem \ref{thm51} to this block matrix leads to 
the desired result.  
\end{proof}

\iffalse 
Similarly, under the same condition, by considering the Hadamard product, i.e., 
\[ \circ^r \begin{bmatrix}A & B \\ B^* & C \end{bmatrix} = 
\begin{bmatrix} \circ^r A & \circ^r B \\ \circ^r B^* & \circ^r C \end{bmatrix} \]
is positive semidefinite, which also leads to analogous results. 

\[ 
\left|\sum_{i=1}^k (a_{ii}c_{ii})^r - \sum_{i=1}^k (b_{ii})^{2r}\right|  \le 
\left(  \sum_{i=1}^k a_{ii}^r\right)\left(  \sum_{i=1}^k c_{ii}^r\right) -
\left(  \sum_{i=1}^k b_{ii}^r\right)^2,  
 \]
and 
\[ 
\sum_{i=1}^k (a_{ii}c_{ii})^r + \sum_{i=1}^k (b_{ii})^{2r}  \le 
\left(  \sum_{i=1}^k a_{ii}^r\right)\left(  \sum_{i=1}^k c_{ii}^r\right) +
\left(  \sum_{i=1}^k b_{ii}^r\right)^2. 
 \]
\fi

\section*{Acknowledgments}
This paper is dedicated to Prof. Weijun Liu (Central South University) on his 60th birthday, 
October 22 of the lunar calendar in 2021. 
I would like to thank Prof. Yuejian Peng 
 for reading carefully through an earlier version 
of this paper. 
This work was supported by  NSFC (Grant No. 11931002).

\end{document}